\documentclass[11pt,a4paper]{amsart}
\usepackage[utf8]{inputenc}
\usepackage{amsmath}
\usepackage{amsfonts}
\usepackage{amssymb}
\usepackage{mathrsfs}
\usepackage{graphicx}
\usepackage{amsthm}
\usepackage{tikz}
\usepackage{tikz-cd}
\usetikzlibrary{patterns}

\theoremstyle{plain}
\newtheorem*{thm*}{Theorem}
\newtheorem{thm}{Theorem}[section]
\newtheorem{prop}[thm]{Proposition}
\newtheorem{cor}[thm]{Corollary}

\newtheorem{lem}[thm]{Lemma}
\theoremstyle{definition}
\newtheorem{de}[thm]{Definition}

\theoremstyle{example}
\newtheorem{ex}[thm]{Example}
\newtheorem{exs}[thm]{Examples}
\newtheorem{rem}[thm]{Remark}

\renewcommand{\phi}{\varphi}
\renewcommand{\rho}{\varrho}
\renewcommand{\epsilon}{\varepsilon}

\title{A brick version of a theorem of Auslander}
\author{Francesco Sentieri}
\date{}

\begin{document}

\begin{abstract}
We prove that a finite dimensional algebra $ \Lambda $ is $ \tau-$tilting finite if and only if all the bricks over $ \Lambda $ are finitely generated.

 This is obtained as a consequence of the existence of proper locally maximal torsion classes for $ \tau-$tilting infinite algebras.
\end{abstract}

\maketitle

\section*{Introduction}

The concept of $ \tau-$tilting module was introduced by Adachi, Iyama and Reiten \cite{tau-orig} to complete, with respect to mutation, the classical theory of tilting modules. 

In their foundational work they established a bijection between functorially finite torsion classes and support $ \tau-$tilting modules. 

Demonet, Iyama and Jasso introduced in \cite{g-vectors} the concept of a $ \tau-$tilting finite algebra: finite dimensional algebras which admit a finite number of basic $ \tau-$tilting modules up to isomorphism.

In the cited work they obtained a connection between bricks, i.e. modules whose endomorphism ring is a division ring, torsion classes and $ \tau-$tilting modules. 

In particular, they proved that a finite dimensional algebra is $ \tau-$tilting finite if and only if it has a finite number of finitely generated bricks up to isomorphism.

This result suggested a strong analogy between representation finite algebras and $ \tau-$tilting finite algebras, leading to new versions of classical results, where the role of indecomposable modules in the representation finite case is now played by bricks.

As an example, in \cite{BTconj} the authors prove the analogue of the first Brauer-Thrall conjecture in this brick setting. 

Following this line, the aim of this paper is to prove a brick version of the following classical result of Auslander:

\begin{thm*}[\cite{AusThm}]
An artin algebra $ \Lambda $ is representation finite if and only if every indecomposable module is finitely generated.
\end{thm*}

To obtain this, we will exploit the concept of torsionfree, almost torsion module, introduced  by  Herzog in \cite{slides} and recently rediscovered by Barnard-Carrol-Zhu \cite{minimIncl} in the finite dimensional setting. 

These modules play an important role within the study of the lattice structure of the set of torsion classes in the category of finitely generated modules over a finite dimensional algebra. 
In \cite{lattices} we find a detailed study of this lattice theoretical properties.

We start in section 1 with a quick reminder about $ \tau-$tilting theory, silting theory and definable classes. We also recall a bijection, essentially due to Crawley-Boevey, between torsionfree classes in the category of finitely generated modules over a noetherian ring and definable torsionfree classes in the large module category.

Section 2 contains a discussion of the concept of torsionfree, almost torsion module, with the new result that torsionfree, almost torsion modules determine definable torsionfree classes.

In Section 3 we observe that, for a non $ \tau-$tilting finite algebra, we can always find a non functorially finite torsion class which is locally maximal  (i.e. without any element covering it in the lattice of torsion classes). 

We use this last observation to give the proof of the following brick version of the theorem of Auslander:

\begin{thm*}
A finite dimensional algebra $ \Lambda $ is $ \tau-$tilting finite if and only if every brick over $ \Lambda $ is finitely generated.
\end{thm*}

\subsection*{Notation} 
$ R $ will denote an arbitrary ring ( with unit ), $ \Lambda $ a finite dimensional algebra over some field $ k $.

 We denote by $ R-$Mod ( resp. $R-$mod)  the category of ( finitely presented) left $ R $ modules. 

For a subclass $ \mathcal{ C } \subseteq R-$Mod, the symbol $ \mathbf{T}(\mathcal{ C }) $, resp. $ \mathbf{F}(\mathcal{ C }) $, denotes the smallest torsion, resp. torsionfree, class in $ R-$Mod containing $ \mathcal{ C } $. 

If $ \mathcal{C} = \{M\} $, we write simply $ \mathbf{T}(M) $ and $ \mathbf{F}(M) $.

We use $ \widetilde{\mathbf{T}} $, resp. $ \widetilde{\mathbf{F}} $, in the same way when referring to torsion pairs in $ R-$mod.

$ \mathbf{ Tors }(R) $, resp. $ \mathbf{ tors }(R) $, is the class of torsion classes in $ R-$Mod, resp. $ R-$mod. 

For a class $ \mathcal{S} \subseteq R-\mathrm{Mod} $, $ \mathrm{Gen}(\mathcal{S}) $ is the subcategory of modules isomorphic to quotients of arbitrary (set-indexed) direct sums of elements in $ \mathcal{S} $. 

Let $ I \subset \mathbb{N} $, denote by $ M^{\perp_I} $  the class of modules $ N $ with $ \operatorname{Ext}^i(M,N) = 0 $, for all $ i \in I $.
\section{Preliminaries}

\subsection{$\tau-$tilting finite algebras}

We recall in this section the definition of a $ \tau-$tilting module by Adachi, Iyama and Reiten \cite{tau-orig} and the related concept of $ \tau-$tilting finiteness, together with some characterizations given in \cite{g-vectors}, \cite{tau-tilt}. 

\begin{de}
Let $ \Lambda $ be a finite dimensional algebra over a field $ k $. Denote by $ \tau $ the Auslader-Reiten translation.

$ M \in \Lambda-\mathrm{mod} $ is $ \tau-$tilting if:
\begin{itemize}
\item[($\tau$1)] $ \mathrm{Hom}(M, \tau M) = 0 $
\item[($\tau$2)] The number of non-isomorphic indecomposable summands of $ M $ is equal to the number of isoclasses of simple $ \Lambda-$modules.
\end{itemize}
Given two $ \tau-$tilting modules $ M_1 $ and $ M_2 $, we say that $ M_1 $ is equivalent to $ M_2 $ if their additive closures in $ \Lambda-\mathrm{mod} $ coincide, $ \mathrm{add}(M_1) = \mathrm{add}(M_2)$.
\end{de}

A module $ M $ is support $\tau-$tilting if there exists some idempotent $ e \in \Lambda $, such that $ M $ is a $\Lambda/(e)-\tau-$tilting module.

\begin{de}
A finite dimensional algebra $ \Lambda $ is $ \tau-$tilting finite if there are only finitely many basic $ \tau-$tilting $ \Lambda-$modules up to isomorphism.
\end{de}

\begin{thm}[{\cite{g-vectors}}]
The following conditions are equivalent for a finite dimensional algebra $ \Lambda $:
\begin{itemize}
\item[(TF1)] $ \Lambda $ is $\tau-$tilting finite
\item[(TF2)] $ \mathbf{tors}(\Lambda) $ is finite.
\item[(TF3)] All torsion classes in $ \Lambda-\mathrm{mod} $ are functorially finite.
\item[(TF4)] There are only finitely many isoclasses of bricks in $\Lambda-\mathrm{mod} $.
\end{itemize}
\end{thm}

For the next characterisation, we recall the concept of silting module, introduced in \cite{siltMod}:

\begin{de}
Let $ R $ be a ring.

Let $ \Sigma $ be a collection of homomorphisms in $ \mathrm{Proj}(R) $, the subcategory of projective modules. We define the class
\[
\mathcal{D}_\Sigma := \{ M \in \Lambda-\mathrm{Mod}\ |\ \operatorname{Hom}(\sigma,\, M)\text{ is surjective for all }\sigma \in \Sigma    \}
\]
A module $ M \in R-\mathrm{Mod} $ is called \emph{silting} if there exists a projective presentation $  P \xrightarrow{\sigma} Q \to M \to 0 $ such that $ \mathrm{Gen}(M) = \mathcal{D}_\sigma $. Notice that in this case $ \mathcal{D}_\sigma $ is a torsion class in $ R-\mathrm{Mod} $. 

Two silting modules are said to be equivalent if they generate the same torsion class.
\end{de}

\begin{prop}[\cite{siltMod}]
A module $ M $ over a finite dimensional algebra is support $ \tau-$tilting if and only if it is a finite dimensional silting module.
\end{prop}

\begin{thm}[\cite{tau-tilt}]
Let $ \Lambda $ be a finite dimensional algebra. TFAE:
\begin{itemize}
\item[(i)] $ \Lambda $ is $ \tau-$tilting finite.
\item[(ii)] For every $ \mathcal{T} \in \mathbf{ Tors }(\Lambda) $ there is a finite dimensional silting module $ S $, with $ \mathcal{T} = \mathrm{Gen}(S) $.
\item[(iii)] Every silting module is finite dimensional up to equivalence.
\end{itemize}
\end{thm}

There are dual notions of cosilting and $\tau^{-}-$tilting modules, with dual properties linked to torsionfree classes. See \cite{siltComm} for details and additional results about cosilting modules.

\subsection{Definable torsion and torsionfree classes}

We collect some well known facts about definable classes and purity. A comprehensive reference can be found in \cite{purity}. 

\begin{de}
A short exact sequence $ 0 \to L \to M \to N \to 0 $ in $ R-\mathrm{Mod} $ is \emph{pure} if for every $ U \in R-\mathrm{mod} $ the sequence
\[
\begin{tikzcd}
0 \arrow[r] & \mathrm{Hom}(U, L) \arrow[r] & \mathrm{Hom}(U, M) \arrow[r] & \mathrm{Hom}(U, N) \arrow[r] & 0
\end{tikzcd}
\]
is an exact sequence of abelian groups. In this case, we say that $ L $ is a \emph{pure submodule} of $ M $.
\end{de}

\begin{de}
Let $ \mathcal{D} \subseteq R-\mathrm{Mod} $. We say that $ \mathcal{D} $ is \emph{definable} if it is closed under products, pure submodules and direct limits.
\end{de}

For finite dimensional algebras, we have the following: 

\begin{prop}[{\cite{Ringel}}]
A short exact sequence $ 0 \to L \to M \to N \to 0 $ in $ \Lambda-\mathrm{Mod} $ is \emph{pure} if for every $ U \in \Lambda-\mathrm{mod} $ the sequence
\[
\begin{tikzcd}
0 \arrow[r] & \mathrm{Hom}(N, U) \arrow[r] & \mathrm{Hom}(M, U) \arrow[r] & \mathrm{Hom}(L, U) \arrow[r] & 0
\end{tikzcd}
\]
is exact in $ \mathrm{Ab}$. 
\end{prop}

\begin{exs}
\label{ex:definable}
\begin{itemize}
\item[(1)] Let $ M \in R-\mathrm{mod} $. Then the torsionfree class $ M^{\perp_0} $ is a definable subcategory of $ R-\mathrm{Mod} $.
\item[(1')] Let $ M \in \Lambda-\mathrm{mod} $. Then, as shown in \cite{dualities}, the torsion class $ {}^{\perp_0}M $ is a definable subcategory of $ \Lambda-\mathrm{Mod} $.
\item[(2)] Given a set-indexed family of definable classes $ \{ \mathcal{D}_i \}_{i \in I} $, the intersection $ \mathcal{D} = \bigcap_{i \in I} \mathcal{D}_i $ is a definable class.
\end{itemize}
\end{exs}

We will be concerned with definable torsion classes and definable torsionfree classes. 

For definable torsionfree classes, we have a result valid for any noetherian ring, using \cite{locallyFp}:

\begin{thm}
\label{thm:CB-bij}
Let $ R $ be a noetherian ring. There is a bijection between torsion pairs in $ R-\mathrm{mod} $ and torsion pairs in $ R-\mathrm{Mod} $ with definable torsionfree class.

This bijection associates to a torsion pair $ ( \mathbf{t}, \mathbf{f} ) $ in $ R-\mathrm{mod} $ the limit closure $ ( \mathcal{T} := \underrightarrow{\lim}\, \mathbf{t},\, \mathcal{F} := \underrightarrow{\lim}\, \mathbf{f}) $. In this setting, $ \mathcal{ T } $ can also be described as $ \mathrm{Gen}(\mathbf{t}) $  and $ \mathcal{F} = \mathbf{t}^{\perp_0} $. 

The inverse of this map sends a torsion pair $ ( \mathcal{T} , \mathcal{F} ) $ to its restriction $ ( \mathcal{T} \cap R-\mathrm{mod}, \mathcal{F} \cap R-\mathrm{mod} ) $.
\end{thm}

A dual version of this bijection holds true in the setting of finite dimensional algebras as a direct consequence of the following observation \cite[Section 2.2]{infDimInRT}:

\begin{lem}
\label{lem:pureSub}
Let $ M \in \Lambda-\mathrm{Mod} $. Then $ M $ is a pure submodule of the product of its finite-dimensional quotients. 
\end{lem}

\begin{proof}
The fact that such a map is a pure monomorphism can be obtained immediately:  every map from $ M $ to a finite dimensional module must factor through a finite dimensional quotient ( Use that the injective cogenerator is finite dimensional to conclude that the map is injective).
\end{proof}

\begin{thm}
\label{thm:CB-dual}
Let $ \Lambda $ be a finite dimensional algebra over a field $ k $. There is a bijection between torsion pairs in $ \Lambda-\mathrm{mod} $ and torsion pairs in $ \Lambda-\mathrm{Mod} $ with definable torsion class.

This bijection associates to a torsion pair $ ( \mathbf{t}, \mathbf{f} ) $ in $ \Lambda-\mathrm{mod} $ the torsion pair cogenerated by $ \mathbf{f}$, $ ( \mathcal{T} := {}^{\perp_0}\mathbf{f},\, \mathcal{F} := ({}^{\perp_0}\mathbf{f})^{\perp_0}) $. In this setting, $ \mathcal{ T } $ can also be described as $ \mathrm{PureCogen}(\mathbf{t}) $, the class of modules obtained as pure submodules of arbitrary direct products of modules in $ \mathbf{t} $.

The inverse of this map sends a torsion pair $ ( \mathcal{T} , \mathcal{F} ) $ to its restriction $ ( \mathcal{T} \cap \Lambda-\mathrm{mod}, \mathcal{F} \cap \Lambda-\mathrm{mod} ) $.
\end{thm}

\begin{proof}
Let $  ( \mathcal{T} , \mathcal{F} )  $ be a torsion pair in $ \Lambda-\mathrm{Mod} $ with definable torsion class. Let $ (\mathbf{t}, \mathbf{f}) $ be the corresponding restriction. 

 We want to show that $ \mathcal{T} = {}^{\perp_0} \mathbf{f} $. The inclusion $ \mathcal{ T } \subseteq {}^{\perp_0} \mathbf{f} $ is immediate. 
 
 For the opposite one, let $ M \in  {}^{\perp_0} \mathbf{f} $. By Lemma \ref{lem:pureSub}, $ M $ can be obtained as a pure subobject of a direct product of finitely generated objects in $ {}^{\perp_0} \mathbf{f} $ which are exactly the modules in $ \mathbf{t} $. 
 
 Whence $ M $ is a pure submodule of a product of objects in $ \mathcal{T} $, so that by definability it follows that $ M \in \mathcal{T} $. 
 
This proves the injectivity of the assignment. Surjectivity is easy noticing that for any torsion pair $ ( \mathbf{t}, \mathbf{f} ) $ in $ \Lambda-\mathrm{mod} $, the torsion pair $ ({}^{\perp_0} \mathbf{f}, \mathbf{F}(\mathbf{f})) $ is an extension to $ \Lambda-\mathrm{Mod}$ with definable torsion pair. 
\end{proof}

Since we will have a continuous interplay between torsion pairs in the small and in the large module category, we fix the following terminology, for a left noetherian ring $ R $:

\begin{de}
Let $ (\mathcal{T}, \mathcal{F}) $ be a torsion pair in $ R-\mathrm{Mod}$ and $ ( \mathbf{t},\mathbf{f}) $ a torsion pair in $ R-\mathrm{mod} $. Then $ (\mathcal{T}, \mathcal{F}) $ \emph{extends} $ ( \mathbf{t},\mathbf{f}) $ if $ \mathbf{t} = \mathcal{T} \cap R-\mathrm{mod} $ and $ \mathbf{f} = \mathcal{F} \cap R-\mathrm{mod}  $. 

In the setting above, we will also say that $ (\mathcal{T}, \mathcal{F}) $ restricts to $ ( \mathbf{t},\mathbf{f}) $.
\end{de}

\begin{ex}
For any torsion pair $ ( \mathbf{t},\mathbf{f}) $ in $ R-\mathrm{mod} $ we have the following two, not necessarily distinct, extensions to $ R-\mathrm{Mod} $:
\begin{itemize}
\item[(S)] The extension with the largest torsion class $ ({}^{\perp_0}\mathbf{f}, \mathbf{F}(\mathbf{f})  ) $
\item[(C)] The extension with the largest torsionfree class $ (\mathbf{T}(\mathbf{t}), \mathbf{t}^{\perp_0}) $
\end{itemize}
\end{ex}

Over a left noetherian ring $ R $, definable torsion classes are parametrized by silting modules \cite[Corollary 3.8]{siltComm}.

Moreover, definable torsionfree classes over an arbitrary ring are parametrized by cosilting modules \cite{cosiltMod}, \cite{cosilt}.

In the following sections we will always work with left noetherian rings, therefore, we write $ \mathcal{T} \in \mathbf{Silt}(R) $ ( resp. $\mathcal{F} \in \mathbf{Cosilt}(R)  $ ) to indicate that $ \mathcal{T} $ is a definable torsion class (resp. $ \mathcal{F} $ is a definable torsionfree  class ) in $ R-\mathrm{Mod} $.

\section{Torsionfree, almost torsion modules}

The concept of torsionfree, almost torsion modules was already being used by Herzog in 2009, as a tool for studying critical summands of cotilting modules. 

More recently, Barnard, Carroll and Zhu rediscovered this concept in their work on the lattice of torsion classes of the category $ \Lambda-\mathrm{mod} $. 

In this section $ R $ will be a left noetherian ring.

\begin{de}[\cite{slides},\cite{minimIncl}]
\label{de:TtF}

Let $ (\mathcal{ T }, \mathcal{F} ) $ be a torsion pair in $ R-$Mod ( resp. in $R-$mod). Let $ B \ne 0 $ a (finitely generated ) $ R $ module. We say that $ B $ is \emph{torsionfree, almost torsion} ( resp. \emph{minimal extending} ) if:
\begin{itemize}
\item[(1)] $ B \in \mathcal{F} $
\item[(2)] Every proper quotient of $ B $ is contained in $ \mathcal{ T } $. 
\item[(3)] For every short exact sequence $ 0 \to B \to F \to M \to 0 $, if $ F \in \mathcal{F} $, then $ M \in \mathcal{F} $.
\end{itemize} 
Condition (2) is equivalent to:
\begin{itemize}
\item[(2')] For any $ F \in \mathcal{F} $ any nonzero morphism $ B \to F $ is a monomorphism
\end{itemize}
Moreover, assuming coditions (1)-(2), we have the following reformulation of (3):
\begin{itemize}
\item[(3')] For every non-split short exact sequence $ 0 \to B \to M \to T \to 0 $, if $ T \in \mathcal{ T } $ then $ M \in \mathcal{ T } $.
\end{itemize}
Dually, we can define \emph{ torsion, almost torsionfree } (and \emph{ minimal co-extending}) modules.
\end{de}

\begin{rem}
Modules satisfying conditions (1)-(2) in the  definition above are precisely the simple torsionfree objects used in Enomoto's study of torsionfree classes \cite{monobrick}.
\end{rem}

\begin{lem}
For a module satisfying (1) and (2),
condition (3) is equivalent to condition (3').
\end{lem}
\begin{proof}
Fix a module $ B \ne 0 $ satisfying (1) and (2) with respect to some torsion pair $ ( \mathcal{ T }, \mathcal{ F } ) $. \\
$``(3) \implies (3')":$ Assume $ B $ satisfies condition (3). Consider a short exact sequence $ 0 \to B \to M \to T \to 0 $, with $ T $ torsion. 

Take the canonical short exact sequence given by the torsion pair $$ 0 \to tM \to M \to M/tM \to 0 $$ with $ tM $ torsion, $ M/tM $ torsionfree. 

Since $ B $ satisfies (2) the map $ g : B \to M \to M/tM $ is zero or injective. Assume it is zero. Then we have the following commutative diagram, obtained with an application of the snake lemma and of the kernel property:
\[
\begin{tikzcd}
0 \ar[r] & B \ar[equal]{d} \ar[r] & tM \ar[ hook, d ] \ar[r] & tM/B \ar[ hook, d] \ar[r] & 0 \\
0 \ar[r] & B \ar[d] \ar[r] & M \ar[ two heads, d ]  \ar[r] & T \ar[ two heads, d] \ar[r] & 0 \\
& 0 \ar[r] & M/tM \ar[r, "\sim"] & M' \ar[r] & 0
\end{tikzcd}
\]
Since $ T $ is torsion, the map $ T \to M' $ must be zero, therefore $ M/tM = 0 $ and $ M $ is torsion.

Assume now that $ g $ is injective. Consider the following commutative diagram:
\[
\begin{tikzcd}
0 \ar[r] & 0 \ar[d] \ar[r] & tM \ar[ hook, d ] \ar[r, "\sim"] & T' \ar[ hook, d] \ar[r] & 0 \\
0 \ar[r] & B \ar[equal, d] \ar[r] & M \ar[ two heads, d ]  \ar[r] & T \ar[ two heads, d] \ar[ r] & 0 \\
0 \ar[r] & B \ar[r] & M/tM \ar[r] & M' \ar[r] & 0
\end{tikzcd}
\]
By (3) the module $ M' $ is both torsion and torsion-free, therefore $ M' = 0 $ and the sequence is split. In conclusion, $ B $ satisfies (3'). \\
$``(3') \implies (3)":$ Assume now $ B $ satisfies (3'). Take a short exact sequence $ 0 \to B \to F \to M \to 0 $, with $ F $ torsionfree. Consider the torsion, torsionfree sequence for $ M $ as above and take the pullback $ F \to M $ along $ tM \to M $ to obtain the commutative diagram:
\[
\begin{tikzcd}
0 \ar[r] & B \ar[r] \ar[equal,d] & M' \ar[hook,d] \ar[r] & tM \ar[hook,d] \ar[r] & 0 \\
0 \ar[r] & B \ar[r] & F \ar[r] & M \ar[r] & 0  
\end{tikzcd}
\]
Since $ M' $ is a submodule of $ F $ it is torsionfree. Now applying the contrapositive of (3') it follows that the sequence above splits, whence $ tM = 0 $ and $ M $ is torsionfree.
\end{proof}

\subsection{Simples in the heart}

Torsionfree, almost torsion modules for any torsion pair $ (\mathcal{T},\mathcal{F}) $ enjoy several orthogonality properties: these follow at once from their characterization as simple objects in a suitable abelian subcategory of $ D(R) $.

Such subcategory is the heart of the t-structure obtained from $ (\mathcal{T},\mathcal{F}) $ by the Happel-Reiten-Smal{\o} construction \cite{HRS-tilt}, which we recall here:

\begin{de}
Let $ (\mathcal{T}, \mathcal{ F }) $ be a torsion pair in a Grothendieck category $ \mathcal{G} $. The classes
\begin{align*}
\mathcal{D}^{\le -1} = & \left\lbrace X^\bullet \in \mathcal{D}^b(\Lambda)\ |\ H^0(X^\bullet) \in \mathcal{T},\ H^i(X^\bullet) = 0,\ \text{for all}\ i > 0   \right\rbrace \\
\mathcal{D}^{\ge 0} = & \left\lbrace X^\bullet \in \mathcal{D}^b(\Lambda) \ |\ H^{0}(X^\bullet) \in \mathcal{F},\ H^i(X^\bullet) = 0,\ \text{for all}\ i < 0   \right\rbrace
\end{align*}
form a t-structure $ (\mathcal{D}^{\le -1}, \mathcal{D}^{\ge 0} )  $ in $ D(\mathcal{G}) $. We denote the corresponding heart by $ \mathscr{H} = \mathcal{D}^{\le -1}[-1] \cap \mathcal{D}^{\ge 0} $.
\end{de}

\begin{rem}
\label{rem:tPairHeart}
Recall that the heart of a t-structure is always an abelian category. 

Moreover, the torsion pair $ (\mathcal{T}, \mathcal{F}) $ induces a torsion pair $ (\mathcal{F}, \mathcal{T}[-1]) $ in the heart $ \mathscr{H} $.
\end{rem}

\begin{thm}[\cite{simples}]
\label{thm:simpHeart}
Let $ (\mathcal{T},\mathcal{F} ) $ be a torsion pair in $ R-\mathrm{Mod} $. 
The simple objects in the HRS-heart $\mathscr{H}$ are precisely the objects $ S $ of the form $ S = F  $ with
$ F $ torsionfree, almost torsion, or $ S = T[-1] $ with $ T $ torsion, almost torsionfree.
\end{thm}

For a proof of the result above, see \cite[Theorem 2.3.6]{RapaThesis}.

\begin{cor}
The collection of torsionfree, almost torsion objects for any torsion pair in $ R-\mathrm{Mod} $ is a semibrick, that is a set of pairwise orthogonal modules whose endomorphism ring is a division ring. The same holds for torsion, almost torsionfree modules.
\end{cor}

We may also obtain some further orthogonality properties of torsionfree, almost torsion and torsion, almost torsionfree modules:

\begin{cor}
\label{cor:orthogonality}
Let $ (\mathcal{ T }, \mathcal{ F }) $ be a torsion pair in $ R-\mathrm{Mod} $. Let $ \mathcal{S} $ be the collection of all the torsion, almost torsionfree modules with respect to this pair. 

Then any torsionfree, almost torsion module is contained in $ \mathcal{S}^{\perp_1} \cap \mathcal{F} $.
\end{cor}

 We end this section by recalling the following result of Parra and Saor\'{i}n \cite{GrothendieckHearts} which we will use later:

\begin{thm}
\label{thm:GrotHeart}
Let $ \mathcal{G} $ be a Grothendieck category and $ (\mathcal{T}, \mathcal{F} ) $ a torsion pair in $ \mathcal{G} $. Then the heart associated with the HRS-tilt at $ (\mathcal{T}, \mathcal{F} ) $ is Grothendieck if and only if $ \mathcal{F} $ is definable.
\end{thm}

\subsection{The small and the large ones}

Since the definition of minimal extending modules for a torsion pair in $ R-$mod is identical to the definition of torsionfree, almost torsion modules for torsion pairs in $R-$Mod, we rightfully expect to have a relation between the two notions.

This is indeed the case:

\begin{prop}
\label{prop:finInf} Let $ R $ be a noetherian ring.
Let $ (\mathbf{ t }, \mathbf{ f }) $ be a torsion pair in $ R-$mod. $ ( \mathcal{T} = \underrightarrow{\lim}\, \mathbf{ t }, \mathcal{F} = \underrightarrow{\lim}\, \mathbf{ f } ) $ the corresponding torsion pair with definable torsionfree class.

Then the minimal extending  modules with respect to $ (\mathbf{ t }, \mathbf{ f }) $  are precisely the finitely generated torsionfree, almost torsion modules for $ (\mathcal{T}, \mathcal{F}) $. 

Moreover, all torsion, almost torsionfree modules for $ (\mathcal{T}, \mathcal{F}) $ are finitely generated and coincide with the minimal co-extending modules for $ (\mathbf{ t }, \mathbf{ f }) $.
\end{prop}

We will need the following lemma in the proof:

\begin{lem}
\label{lem:finTfT}
Let $ (\mathcal{ T }, \mathcal{ F }) $ be a torsion pair with $ \mathcal{F} \in \mathbf{Cosilt}(R) $. Then all the torsion, almost torsionfree modules are finitely generated. 

Over a finite dimensional algebra $\Lambda $, we can say dually that for a torsion class in $ \mathbf{Silt}(\Lambda) $ all torsionfree, almost torsion modules are finite dimensional.
\end{lem}

\begin{proof}
By Theorem \ref{thm:CB-bij} the torsion pair can be written as $ ( \mathrm{Gen}(\mathbf{t}) , \mathbf{t}^{\perp_0} ) $ with $ \mathbf{t} $ torsion class in $ R-$mod. 

Assume $ T $ is torsion, almost torsionfree and $ T \not \in R-\mathrm{mod}$. By assumption, all proper submodules of $ T $, in particular all possible images of morphisms from a finitely generated module, must be torsionfree, whence $ T \in \mathbf{ t }^{ \perp_0 } $, since $ \mathbf{ t } $ is a torsion class. This yields a contradiction.

For the second case proceed dually using Theorem \ref{thm:CB-dual}.
\end{proof}

Now we can prove the proposition:

\begin{proof}
 It is clear that the finitely generated torsionfree, almost torsion modules for $ (\mathcal{T}, \mathcal{F}) $  are minimal extending for  $ (\mathbf{ t }, \mathbf{ f }) $ . 

So suppose that $ S $ is minimal extending. We have immediately $ S \in \mathcal{F} $ and that all its proper quotients are in $ \mathbf{ t } $, so in particular in $ \mathcal{T} $.

The only condition that we have to check is the last one: consider a short exact sequence
\[
\begin{tikzcd}
0 \ar[r] & S \ar[r,"f"] & F \ar[r] & M \ar[r] & 0
\end{tikzcd}
\]
with $ F \in \mathcal{ F } $. Recall that $ \mathcal{ F } = \mathbf{ t }^{\perp_0} $. Hence suppose we have a map $ T \to M $ with $ T \in \mathbf{ t } $ which we may assume w.l.o.g. injective; taking the pullback along $ F \to M $ we obtain the commutative diagram:
\[
\begin{tikzcd}
0 \ar[r] & S \ar[r] \ar[d,equals] & P \ar[r] \ar[d] & T \ar[r] \ar[d] & 0 \\
0 \ar[r] & S \ar[r,"f"] & F \ar[r] & M \ar[r] & 0
\end{tikzcd}
\]
in $ R-$mod, where by the minimal extending property (3') of $ S $, $ P $ must be torsion or the sequence must split. Since $ P $ is a submodule of $ F $ it must be torsionfree, so the sequence splits, and $ T = 0 $, proving that $ M \in \mathbf{ t }^{\perp_0} $.

The proof of the second statement is more involved, as the available description of $ \mathcal{T} $ is less practical to work with. 

We need to check that every minimal co-extending module is torsion, almost torsionfree. 

Let $ S $ be minimal co-extending. 
By definition $ S \in \mathbf{t} \subseteq \mathcal{T} $.

 The second property is immediately verified, as every proper submodule of $ S $ is an element of $ \mathbf{f} \subseteq \mathcal{F} $.

For the third property, consider a short exact sequence:
\[
\begin{tikzcd}
0 \ar[r] & K \ar[r] & T \ar[r, "f"] & S \ar[r] & 0
\end{tikzcd}
\]
with $ T \in \mathcal{ T } $. By Theorem \ref{thm:CB-bij} we can find a family $ \{U_i\}_{i \in I} $ of objects of $ \mathbf{t} $, with an epimorphism $ h : \coprod_I U_i \to T $. 

Since $ S $ is finitely generated, we can find a finite subset $ I_0 \subseteq I $ such that $ \tilde{f} = f \circ h \circ \iota_{I_0} $ is surjective (where $ \iota_{I_0} : \coprod_{I_0} U_i \to \coprod_I U_i $  ). 

By definition, $ \coprod_{I_0} U_i \in \mathbf{t} $, whence  $ K' = \ker \tilde{f} $ is a torsion module, as $ S $ is minimal co-extending.

Consider the following pullback diagram:

\[
\begin{tikzcd}
         &                              & 0  & 0    \\
0 \ar[r] & K \ar[r] \ar[d, equals] & T \ar[r, "f"] \ar[u] & S \ar[r] \ar[u] & 0 \\
0 \ar[r] & K \ar[r]  & P \ar[r] \ar[u] & \coprod_{I_0} U_i \ar[r] \ar[u,"\tilde{f}"] & 0 \\
0 \ar[r] & 0 \ar[r] \ar[u] & K' \ar[r, equals] \ar[u] & K' \ar[r] \ar[u] & 0 
\end{tikzcd}
\]
Notice that $ P \in \mathcal{ T } $ since $ T $ and $ K' $ are torsion modules. 

Then consider the map $ h \circ i_{I_0} : \coprod_{I_0} U_i \to T $, using the pull-back property we obtain the following commutative diagram:
\[
\begin{tikzcd}
   & T \ar[r, "f"]    & S \\
   & P \ar[u] \ar[r]  & \coprod_{I_0} U_i \ar[u,"\tilde{f}"'] \\
\coprod_{I_0} U_i \ar[uur, "h \circ i_{I_0}" ] \ar[ur,dotted,"n"] \ar[urr, equals]
\end{tikzcd}
\]
Thus the middle short exact sequence in the first diagram splits, hence $ K \in \mathcal{ T } $.
\qedhere
\end{proof}

\subsection{Torsionfree, almost torsion and functorially finite classes}
Starting from this section, we will always work over a finite dimensional algebra $ \Lambda $. 
Recall the following definition:

\begin{de}
A subcategory $ \mathbf{s} $ of $ \Lambda-\mathrm{mod} $ is \emph{functorially finite} if every module admits a left and right $ \mathbf{s}-$approximation. 
\end{de}

For a torsion pair $ (\mathbf{t}, \mathbf{f} ) $ in $ \Lambda-\mathrm{mod} $ we know by work of Smal{\o} \cite{tPairTilt} that  $ \mathbf{t} $ is functorially finite if and only if $ \mathbf{f} $ is functorially finite. 

Functorially finite torsion pairs are parametrized by support $ \tau-$tilting modules:

\begin{thm}[ {\cite[Theorem 2.7]{tau-orig} } ]
\label{thm:tauFF}
There is a bijection between functorially finite torsion classes and equivalence classes of support $ \tau-$tilting modules, which associates to any representative $ T $ the torsion class $ \widetilde{\mathbf{T}}(T) = \mathrm{Gen}(T) $.
\end{thm}

Such classes admit a unique extension to $ \Lambda-\mathrm{Mod} $ as was already observed in  \cite[Proposition 5.3]{QvS}. We give a proof for completeness.

\begin{prop}[Vit\'{o}ria]
\label{prop:ffUniqueExt}
Let $ (\mathbf{t}, \mathbf{f} ) $ be a torsion pair in $ \Lambda-\mathrm{mod}$. Then the following statements are equivalent:

\begin{itemize}
\item[(1)] $ \mathbf{t} $ is functorially finite.
\item[(2)] There exists a torsion pair $ ( \mathcal{ T }, \mathcal{ F }) $ in $ \Lambda-\mathrm{Mod}$ extending $( \mathbf{t}, \mathbf{f} )$, such that both the torsion and the torsionfree classes are definable.
\item[(3)] There is a unique torsion pair extending $ ( \mathbf{t} , \mathbf{f} ) $ to $ \Lambda-$Mod.
\end{itemize}
\end{prop}

\begin{proof}
"(1) $\implies$ (2)" : By Theorem \ref{thm:tauFF} we find a support $ \tau-$tilting module $ T $ generating $ \mathbf{t} $. The torsion pair $ ( \mathbf{T}(T), T^{\perp_0} ) $ in $ \Lambda-\mathrm{Mod} $ extends the original torsion pair and it is definable on both sides.

"(2) $\implies$ (1)" : By assumption we have a torsion pair $ ( \mathcal{ T }, \mathcal{ F }) $ extending $ (\mathbf{t}, \mathbf{f} ) $ to $ \Lambda-$Mod, such that both the torsion and the torsionfree classes are definable. Using Theorem \ref{thm:CB-bij}, we obtain that $ \mathcal{ T } = \underrightarrow{\lim}(\mathbf{t} ) $. 

 Then we can apply a result of Lenzing, see as a reference \cite[Corollary 3.4.37]{purity}, ensuring that the torsion class $ \mathbf{ t } $ has left  approximations, hence it is functorially finite (every torsion class provides right approximations).

"(2) $ \implies $ (3)" : Using Theorems \ref{thm:CB-bij} and \ref{thm:CB-dual} it follows that $ \mathcal{T} = {}^{\perp_0}\mathbf{f} $ and $ \mathcal{F} = \mathbf{t}^{\perp_0} $. In particular, this torsion pair is both the smallest and the largest torsion pair extending $ (\mathbf{t}, \mathbf{f}) $, hence the unique one.

"(3) $ \implies $ (2)" : Immediate, using the fact that the extension with definable torsion class and the one with definable torsionfree class must coincide.
\end{proof}

\begin{cor}
\label{cor:finDimTtF-TfT}
Let $ (\mathbf{t}, \mathbf{f}) $ be a functorially finite torsion pair in $ \Lambda-\mathrm{mod} $.

Then for its unique extension $ ( \mathcal{T}, \mathcal{F}) $ to $ \Lambda-\mathrm{Mod} $ all the torsionfree, almost torsion and torsion, almost torsionfree modules are finite dimensional. 
\end{cor}

\begin{proof}
$ \mathcal{F}  \in \mathbf{Cosilt}(\Lambda) $ and $ \mathcal{T} \in \mathbf{Silt}(\Lambda) $ by Proposition \ref{prop:ffUniqueExt}. This yields the result by means of Lemma \ref{lem:finTfT}.
\end{proof}

\subsection{Torsionfree, almost torsion modules determine torsion pairs}

We want to prove that a definable torsionfree class $ \mathcal{F} $ in $ \Lambda-\mathrm{Mod} $ is uniquely determined by its torsionfree, almost torsion modules.

This is a consequence of the following lemma, ensuring the existence of finitely generated, even finitely presented, objects in the heart associated with $ \mathcal{F} $ and of the characterization of torsionfree, almost torsion modules as simple objects in the heart.

\begin{lem}[Parra-Saorìn-Virili]
\label{lem:finInHeart}
Any finite dimensional module in $ \mathcal{ F} $ is finitely presented when seen as an object in the heart $ \mathscr{H} $.
\end{lem} 

\begin{rem}
The original result \cite[Corollary 4.3]{finiteHearts} is much stronger and gives a complete description of finitely presented objects in the torsion class $ \mathcal{F} $ in $ \mathscr{H} $ in terms of properties of the corresponding modules. 
\end{rem}

At this point, we can prove:

\begin{cor}
\label{cor:TfExist}
Consider a definable torsionfree class $ 0 \ne \mathcal{F} \in \mathbf{Cosilt}(\Lambda) $, then there exists some (not necessarily finite dimensional) torsionfree, almost torsion module.
\end{cor}

\begin{proof}
Since  $ 0 \ne \mathcal{F} $, by Theorem \ref{thm:CB-bij} there exists some non-zero finite dimensional module $ F \in \mathcal{F} $. $ F $ is a non-zero finitely generated object in $ \mathscr{H} $, which is a Grothendieck category by Theorem \ref{thm:GrotHeart}, whence $ F $ has a maximal subobject, giving rise to a simple quotient $ S $ which is an element of $ \mathcal{ F } $, as this is a torsion class in the heart (  see Remark \ref{rem:tPairHeart} ). 

By Theorem \ref{thm:simpHeart}, this simple object $ S $ corresponds to a torsionfree, almost torsion module for $ \mathcal{F} $.
\end{proof}

\begin{rem}
\label{rem:IvoCrit}
A different proof for the existence of torsionfree, almost torsion modules can be obtained using model theoretic arguments, related to the link with critical summands of cotilting modules \cite{slides}.
\end{rem}

Recall that a finite dimensional module $ S \ne 0 $ is simple with respect to a torsionfree class if it satisfies conditions (1) - (2) in definition \ref{de:TtF}.

 Notice that a simple module for some torsionfree class $ \mathbf{f} $ in $ \Lambda-\mathrm{mod} $ automatically satisfies conditions (1) - (2) for any extension of $ \mathbf{f} $ to $ \Lambda-\mathrm{Mod} $.

\begin{prop}
\label{prop:betaInj}
Let $ \mathcal{F} $, $ \mathcal{F}' $ be  definable torsionfree classes in $ \Lambda-\mathrm{Mod} $ such that the torsionfree, almost torsion modules for $ \mathcal{F} $ and $ \mathcal{F}'$ coincide. Then $ \mathcal{F} = \mathcal{F}'$.

Moreover, every finite dimensional simple object for $ \mathcal{F}$ embeds in some torsionfree, almost torsion module.
\end{prop}

\begin{proof}

Since the torsionfree classes are definable, by Theorem \ref{thm:CB-bij}, it is enough to prove that $ \mathcal{F}\, \cap\, \Lambda-\mathrm{mod} =  \mathcal{F}' \cap \Lambda-\mathrm{mod} $. 

By \cite[Theorem 3.15]{monobrick}, we know that torsionfree classes in $ \Lambda-\mathrm{mod} $ are determined by their simple objects. 

Assume without loss of generality that $ \mathcal{F} \ne 0 $.

Let $ B \in \mathcal{F} \cap \Lambda-\mathrm{mod} $ be a simple object in the torsionfree class (for instance take any simple, in the usual sense, submodule of some nonzero module in $\mathcal{F} $ ). 

Then, by Lemma \ref{lem:finInHeart}, $ B $ is a finitely generated object in the heart associated with $ \mathcal{F} $, whence it admits a simple quotient $ S $.

In the heart, $ B $ is an object of the torsion class $ \mathcal{F} $, therefore $ S \in \mathcal{F}$. 

It follows, by Theorem \ref{thm:simpHeart}, that there is a torsionfree, almost torsion module $ S $ for $ \mathcal{F}$, together with a non-zero homomorphism $ f : B \to S $. 

Since $ B $ is simple in the torsionfree class, $ f $ must be a monomorphism, whence $ B $ is a submodule of $ S $ which is by assumption an element of $ \mathcal{F}' $.

 This procedure shows that all the simple objects in $ \mathcal{F} $  are contained in $ \mathcal{F}'$.

Whence $ \mathcal{F} \cap \Lambda-\mathrm{mod} \subseteq \mathcal{F}' \cap \Lambda-\mathrm{mod} $. The same argument works for the opposite inclusion.
\end{proof}

\begin{rem}
Minimal extending modules ( that is finite dimensional torsionfree, almost torsion modules ) do not provide enough information to determine a torsionfree class in $ \Lambda-\mathrm{mod} $ as already observed in \cite{wideTorsion}, \cite{monobrick}.

This shows that the role of infinite dimensional torsionfree, almost torsion modules is not negligible in the study of torsion pairs of $ \Lambda-\mathrm{mod} $.
\end{rem}

\section{The lattice of torsion classes}

There is a natural partial order on the collection of torsion classes $ \mathbf{tors}(\Lambda) $ of $ \Lambda-\mathrm{mod} $ given by inclusion. 

As shown in \cite{lattices} the resulting poset has the structure of a complete lattice and enjoys several nice lattice-theoretic properties. 

More explicitely, we have the following description of the meet and join of a set indexed family $ \{ \mathbf{t}_i \}_{i \in I} $ of torsion classes:

\[
\bigwedge_{I} \mathbf{t}_i := \bigcap_I \mathbf{t}_i \quad, \quad \bigvee \mathbf{t}_i := \widetilde{\mathbf{T}}\Big( \bigcup_I \mathbf{t}_i \Big)
\]

We recall also some basic lattice theoretic terminology:

\begin{de}
Let $ (L, \le ) $ be a poset, $ x, y \in L $: 
\begin{itemize}
\item[(1)] The \emph{interval} $ [ x, y ] $ is the poset supported by those $ z \in L $ with $ x \le z \le y $. Notice that if $ L $ is a (complete) lattice, any non-empty interval in $ L $ is a (complete) sublattice of $ L $.
\item[(2)] We say that $ y $ \emph{covers} $ x $ if $ x < y $ and for any $ z \in L $ such that $ x \le z \le y $, either $ z = x $ or $ z = y $.
\item[(3)] Let $ L $ be a complete lattice. Then $ x $ is \emph{completely meet irreducible} if whenever $ x = \bigwedge_{I} y_i $, with $ y_i \in L $, we must have $ x = y_j $ for some $ j \in I $. 

This condition can be restated as follows: there is a unique element $ x^\star $ covering $ x $, and for every $ y > x $ we have $ y \ge x^\star $.
\end{itemize}
\end{de}

We recall a result proven in \cite{minimIncl}, relating minimal extending modules with the covering relation in $ \mathbf{tors}(\Lambda) $. 

\begin{thm}[{\cite[Theorem 1.0.2]{minimIncl}}]
\label{thm:minExt}
Let $ \mathbf{ t } \in \mathbf{tors}(\Lambda) $, $ \mathcal{S} $ be a collection of representatives of the isoclasses of minimal extending modules for $\mathbf{t}$. 

Then the elements of $ \mathcal{S}$ are in bijection with torsion classes covering $ \mathbf{t} $.
\end{thm}

A typical phenomenon for the lattice of torsion pairs in the $ \tau-$tilting infinite case is the presence of non-trivial locally maximal elements.

\begin{de}
Let $ \mathbf{ t } \in \mathbf{tors}(\Lambda) $. We say that $  \mathbf{ t } $ is \emph{locally maximal} if there are no elements of $ \mathbf{ tors }(\Lambda) $ covering $ \mathbf{t} $. 
\end{de}

\begin{rem}
Any locally maximal torsion class is obtained as the meet of all the strictly larger torsion classes. In particular, such classes are never completely meet irreducible.

Also, notice that there is a unique functorially finite locally maximal element, namely the torsion class $ \Lambda-\mathrm{mod} $, which is by definition the meet of the empty set. 

In fact, for any functorially finite torsion pair $ \mathbf{t} $ properly contained in some other class $ \mathbf{u} $, it is possible, by means of mutation, to find a class $ \mathbf{t}^{\star} $ covering  $ \mathbf{t} $ such that $ \mathbf{t}^{\star} \le \mathbf{u} $. See \cite{g-vectors}.
\end{rem}

\begin{lem}
Let $ \mathbf{ t } \in \mathbf{ tors }(\Lambda) $ be a meet irreducible, but not completely meet irreducible element, then $ \mathbf{ t } $ is locally maximal.
\end{lem} 

\begin{proof}
Assume by contradiction it has some covering class. By meet irreducibility it has precisely one, say $\mathbf{ t }^\star $. 

As $ \mathbf{t} $ is not completely meet-irreducible, but it has just one covering class, there must be some torsion class $ \mathbf{ u } \supsetneq \mathbf{ t } $ such that $ \mathbf{ u } \not \supseteq \mathbf{ t }^\star $. 

But this is absurd, since $ \mathbf{ u } \wedge \mathbf{ t }^\star = \mathbf{ t } $.
So $ \mathbf{ t } $ is locally maximal.
\end{proof}

\begin{ex}
We discuss the most common example of a torsion pair without minimal extending modules. Let $ k $ be an algebraically closed field.

Let $ \Lambda = kK_2 $ the Kronecker algebra, obtained as the path algebra of the quiver \begin{tikzcd}
0 \ar[r, shift left=2] \ar[r] & 1
\end{tikzcd}. 

This is a finite dimensional hereditary algebra, as such, we have that any indecomposable in $ \Lambda-\mathrm{mod} $ is contained in the preprojective $ \mathbf{p} $, regular $ \mathbf{r} $ or preinjective $ \mathbf{q} $ component of the AR-quiver. 

Recall that the additive closure of the regular component $ \mathbf{r} $ is a wide subcategory of $ \Lambda-\mathrm{mod} $ whose simple objects are called simple regular modules.

The torsion class generated by the modules in the regular component $ \widetilde{\mathbf{T}}(\mathbf{r}) $ contains all the regular and preinjective modules, but no preprojective module. 

The preprojective component contains a countable collection of bricks $ P_i $, such that $ \widetilde{\mathbf{T}}(P_i) \supset \widetilde{\mathbf{T}}(P_{i + 1}) $. Moreover, $ \bigcap \widetilde{\mathbf{T}}(P_i) = \widetilde{\mathbf{T}}(\mathbf{r} ) $. 

Any torsion class larger than $ \widetilde{\mathbf{T}}(\mathbf{r}) $ is of the form $ \widetilde{\mathbf{T}}(P_i) $ for some $ P_i $, so it follows that $ \widetilde{\mathbf{T}}(\mathbf{r}) $ is locally maximal. 

\begin{figure}
\centering
\begin{tikzpicture}[scale=0.70]
\fill [black!20, rounded corners] (2, -2) -- (2, 7) -- (15,5.5) -- (17.5,-2.2) -- cycle;
\fill [pattern=crosshatch,pattern color=black!35, rounded corners] (2, -2) -- (2, 7) -- (15,5.5) -- (17.5,-2.2) -- cycle;
\fill [black!50, rounded corners] (5.6, -2) -- (5.6, 6) -- (15,5) -- (17.2,-2) -- cycle;
\draw[->] (0,0) node[anchor=north east]{$P_0$} -- (1,1) node[anchor=south west]{$P_1$};
\draw[xshift=-2pt, yshift=3pt, ->] (0,0) -- (1,1);
\draw[->] (1.7,1) -- (2.5,0);
\draw[xshift=3pt, yshift=2pt, ->] (1.7,1) -- (2.5,0);
\draw[->] (3,0) node[anchor=north east]{$P_2$} -- (4,1) node[anchor=south west]{$P_3$};
\draw[xshift=-2pt, yshift=3pt, ->] (3,0) -- (4,1);
\draw (5,0.5) node{$\dots$};
\draw (6,-1) -- (6,5);
\draw (6.5,-1) -- (6.5, 5);
\draw (6.25, -1) ellipse(7pt and 2pt);
\draw (7,-1) -- (7,5);
\draw (7.5,-1) -- (7.5, 5);
\draw (7.25, -1) ellipse(7pt and 2pt);
\draw (8, 0.5) node{$\dots$};
\draw (8.5,-1) -- (8.5,5);
\draw (9,-1) -- (9, 5);
\draw (8.75, -1) ellipse(7pt and 2pt);
\draw (10, 0.5) node{$\dots$};
\draw[->] (11.7,1) -- (12.5,0);
\draw[xshift=3pt, yshift=2pt, ->] (11.7,1) -- (12.5,0);
\draw[->] (13,0) node[anchor=north east]{$Q_2$} -- (14,1) node[anchor=south west]{$Q_1$};
\draw[xshift=-2pt, yshift=3pt, ->] (13,0) -- (14,1);
\draw[->] (14.7,1) -- (15.5,0);
\draw[xshift=3pt, yshift=2pt, ->] (14.7,1) -- (15.5,0) node[anchor=north west]{$Q_0$} ;
\draw (3.8, 5.5) node{$\widetilde{\mathbf{T}}(P_2)$};
\draw (10, 4.5 ) node{$ \widetilde{\mathbf{T}}(\mathbf{r}) $};
\end{tikzpicture}
\caption{Torsion pairs in $ kK_2-\mathrm{mod} $ }
\end{figure}

Consider now the corresponding definable torsion pair in $ \Lambda-\mathrm{Mod} $. By Corollary \ref{cor:TfExist}, we know that there must be some torsionfree, almost torsion module for this torsion pair, which is necessarily infinite dimensional as the corresponding torsion class in $ \Lambda-\mathrm{mod} $ has no minimal extending modules. 

It is easy to compute the torsion, almost torsionfree modules for this torsion pairs: they are precisely the simple regular modules $ \{S_\lambda\}_{\Lambda} $. 

By Corollary \ref{cor:orthogonality}, it follows that the torsionfree, almost torsion modules must lie in the orthogonal category $ (\{ S_{\lambda}\}_{\Lambda})^{\perp_{0,1}} $ which is known to be equivalent to the module category $ k(X)-\mathrm{Mod} $. 

Such a subcategory contains a unique brick, up to isomorphism, since $ k(X) $ is a field, and this brick is the generic module $ G $ described by Ringel \cite{Ringel}. 

Since a torsionfree, almost torsion module must exist, we conclude that $ G $ is the unique torsionfree, almost torsion module for the extended torsion pair.
\end{ex}

\subsection{Locally maximal classes for $ \tau-$tilting infinite algebras}

We need some preparations to show the existence of non functorially finite locally maximal elements:

\begin{de}
Let $ L $ be a complete lattice. An element $ x \in L $ is \emph{compact} if for every set-indexed family $ \{y_i\}_{i \in I} $ such that 
$ x \le \bigvee_{i \in I} y_i $
there exists a finite subset $ J \subseteq I $ such that $ x \le \bigvee_{i \in J} y_i $. Dually we have the notion of a \emph{co-compact} element.
\end{de}

We will use the following observation ( the contrapositive of {\cite[Lemma 3.10]{g-vectors} } ). 

\begin{lem}[]
\label{lem:compactness}
Let $ \mathbf{ t } \in \mathbf{ tors }(\Lambda) $ be a functorially finite torsion class. Let $ \{ \mathbf{ t }_i \} $ be a chain of torsion classes indexed by some ordinal.

If $ \mathbf{ t } = \bigvee_i \mathbf{ t }_i $, then there exists some $ j $ such that $ \mathbf{t}_j = \mathbf{t} $.
\end{lem}

\begin{proof}
Any functorially finite torsion class is both compact and co-compact \cite[Proposition 3.2]{lattices}, in particular, there exists a finite subchain $ \mathbf{ t }_{i_n} $ such that  $ \mathbf{ t } = \bigvee_n \mathbf{ t }_{i_n} $. Whence, $ \mathbf{t}_j = \mathbf{t} $ for some $ j $. 
\end{proof}

\begin{lem}
Let $ \mathbf{ u }_1, \mathbf{ u }_2 $ be functorially finite torsion classes in  $ \Lambda $.

 $ I = [ \mathbf{ u }_1, \mathbf{ u }_2 ] \subseteq \mathbf{ tors }(\Lambda) $ the corresponding interval. 

Then if $ I $ has not finite length, it contains a maximal and a minimal non functorially finite torsion class $ \mathbf{t}_{max} $ and $ \mathbf{ t }_{min} $. 

Moreover, the first class is meet irreducible in $ I $ but not completely meet irreducible, while the second class is join irreducible in $ I $ but not completely join irreducible.
\end{lem}

\begin{proof}
We denote by $ \mathbf{nftors}(\Lambda) \subset \mathbf{tors}(\Lambda) $ the poset of non functorially finite torsion classes.

By assumption, $ I $ must contain either an infinite strictly ascending chain or an infinite strictly decreasing chain. 

Since $ I $ is a complete sublattice, the join of the first chain, or the meet of the second one yields a non functorially finite class lying in $ I $ ( using compactness, or co-compactness of functorially finite classes ),  proving that the poset $ nI = I \cap \mathbf{nftors}(\Lambda) $ is not empty. 

This poset and its dual satisfy the hypotheses of Zorn's lemma, in fact for any chain in $ nI $ the join, or meet, of the chain in $ I $ is again a non-functorially finite torsion class by Lemma \ref{lem:compactness} and its dual, giving the required upper, or lower, bound.

So we conclude that $ nI $ has a maximal and a minimal element. 

Now, if $ \mathbf{ t } $ is such a maximal element, then starting with the obvious inclusion $ \mathbf{ t } \subsetneq \mathbf{u}_2 $ and applying inductively \cite[Theorem 3.1]{g-vectors} it is possible to construct an infinite descending chain of functorially finite torsion classes $ \mathbf{ t } \subseteq ... \subsetneq \mathbf{ t }_n \subsetneq \dots \subsetneq \mathbf{ t }_1 \subsetneq \mathbf{ t }_0 = \mathbf{u}_2 $.

Now, by co-compactness the meet of an infinite strictly descending chain is not functorially finite, whence we can conclude by maximality that $ \mathbf{ t } = \bigwedge_{ i \in \mathbb{N}} \mathbf{t}_i $, proving that it is not completely meet irreducible. 

Assume now that $ \mathbf{ t } = \mathbf{ s}_1 \wedge \mathbf{s}_2 $ , for some $ \mathbf{ s }_i \in I $.

By definition of meet, $ \mathbf{ t } \le \mathbf{ s }_i $, so if any of the two is not functorially finite, we must have equality, by maximality in $ nI $. 

So assume they are both functorially finite. By the argument above, $ \mathbf{ t } =  \bigwedge_{ i \in \mathbb{N}} \mathbf{t}_i $, but by co-compactness of $ \mathbf{s}_i $ there is some index $ j $, such that $ \mathbf{ t }_j \le \mathbf{s}_1, \mathbf{ s }_2 $, but this is a contradiction, since $ \mathbf{t}_j > \mathbf{t} $. So maximal non functorially finite elements are meet irreducible in $ I $.

Dual arguments yield the dual results.  
 \end{proof}

\begin{cor}
\label{cor:locMaxExists}
Let $ \Lambda $ be a not $ \tau-$tilting finite algebra, then there exists a maximal non functorially finite torsion class. Such torsion class is meet irreducible, but not completely meet irreducible, hence locally maximal. 
\end{cor}
 
\begin{proof}
Apply the Lemma above to the interval $ [0, \Lambda-\mathrm{mod} ] $ which has infinite length, see \cite[Proposition 3.9]{g-vectors}, to obtain a maximal element $ \mathbf{t} $ in $ \mathbf{nftors}(\Lambda) $ with the required properties. 
 \end{proof}
 
 \subsection{The main theorem}
 
 We need a last lemma before proceeding into the proof of the main theorem. This construction is already present in the literature, see \cite{minimIncl}, we give a proof for the convenience of the reader:
 
 \begin{lem}
 \label{lem:tTFinGen}
 Let $ B \in \Lambda-\mathrm{Mod} $ be a brick. Then $ B $ is the unique torsion, almost torsionfree module for the torsion pair $ ( \mathbf{T}(B), B^{\perp_0}) $. 
 \end{lem} 
 
 \begin{proof}
 We check the three conditions dual to those in Definition \ref{de:TtF}:
 
 (1) $ B \in \mathbf{T}(B) $ by definition.
 
 (2) Since $ B $ is a brick, for every proper submodule $ M $ of $ B $ we must have $ M \in B^{\perp_0} $, that is $ M $ is torsionfree.
  
 (3') Consider a short exact sequence:
  \[
  \begin{tikzcd}
  0 \ar[r] & F \arrow[r] & M \arrow[r,"f"] & B \ar[r] & 0
  \end{tikzcd}
  \]
  with $ F \in B^{\perp_0} $. If $ M \not \in B^{\perp_0} $, let $ 0 \ne g : B \to M $. Since $ F $ is torsionfree, $ g $ can not factor through $ F $, in particular $ f \circ g \ne 0 $. 
  
  Since $ B $ is a brick this endomorphism must be invertible, which means that the sequence splits. 
  This proves that $ B $ is torsion, almost torsionfree. 
  
  Any other torsion, almost torsionfree module $ S $, if not isomorphic to $ B $, would be orthogonal to it, in particular torsionfree. This is a contradiction, yielding uniqueness (up to isomorphism).
 \end{proof}

\begin{lem}
Let $ \Lambda $ be a $ \tau-$tilting finite algebra. Then every brick in $ \Lambda-\mathrm{Mod} $ is finite dimensional.
\end{lem}

\begin{proof}
Let $ B $ be a brick. By Lemma \ref{lem:tTFinGen}, the module $ B $ is torsion, almost torsionfree with respect to $  ( \mathbf{T}(B), B^{\perp_0}) $.

The restriction of $ ( \mathbf{T}(B), B^{\perp_0}) $ to $ \Lambda-\mathrm{mod} $ is necessarily functorially finite, as all torsion classes in $ \Lambda-\mathrm{mod} $ are functorially finite by hypothesis. 

By Proposition \ref{prop:ffUniqueExt},  $ ( \mathbf{T}(B), B^{\perp_0}) $ is the unique extension of the functorially finite torsion pair obtained above.

Whence, by Corollary \ref{cor:finDimTtF-TfT}, all the torsion, almost torsionfree and torsionfree, almost torsion modules for $  ( \mathbf{T}(B), B^{\perp_0}) $ are finite dimensional. 

In particular the brick $ B $ is finite dimensional.   
\end{proof}

\begin{lem}
Let $ \Lambda $ be a $ \tau-$tilting infinite algebra. Then there exists some infinite dimensional brick in $ \Lambda-\mathrm{Mod} $.
\end{lem}

\begin{proof}
Apply Corollary \ref{cor:locMaxExists} to obtain a locally maximal non functorially finite torsion class $ \mathbf{t} $ in $ \Lambda-\mathrm{mod}$.

By Theorem \ref{thm:minExt}, the corresponding torsion pair $ ( \mathbf{t},\mathbf{f} ) $ has no minimal extending modules. 

Consider now the definable torsion pair extending it, 
$$ ( \mathcal{T} = \underrightarrow{\lim}\, \mathbf{ t }, \mathcal{F} = \underrightarrow{\lim}\, \mathbf{ f } ). $$ By Corollary \ref{cor:TfExist}, there is some torsionfree, almost torsion module $ B $ for this torsion pair.

If $ B $ were finite dimensional, by Proposition \ref{lem:finTfT}, it would be minimal extending for the original torsion pair $ (\mathbf{t},\mathbf{f}) $ which gives a contradiction. 
\end{proof}

Combining the two lemmas we can finally obtain:

\begin{thm}
A finite dimensional algebra $ \Lambda $ is $ \tau-$tilting finite if and only if every brick over $\Lambda$ is finitely generated.
\end{thm}

\section*{Acknowledgments}

The author would like to thank his supervisor Lidia Angeleri H\"{u}gel for her guidance and support, and Rosanna Laking for many hours of interesting discussion on cosilting theory and related topics. 

This paper is part of the author's PhD work, funded by the Universities of Trento and Verona.


\begin{thebibliography}{99}
\bibitem[AIR]{tau-orig} T. Adachi, O. Iyama, I. Reiten; \emph{$\tau-$tilting theory}, Compos. Math. 150 (2014), 415–452.
\bibitem[AHe]{simples} L. Angeleri H\"{u}gel, I. Herzog; \emph{Simples in the heart of t$-$structures and cotilting modules, } unpublished notes. 
\bibitem[AHr]{siltComm} L. Angeleri H\"{u}gel, M. Hrbek; \emph{Silting modules over commutative rings},
International Mathematics Research Notices 2017(13) (2017), 4131–4151.
\bibitem[AMV1]{siltMod} L. Angeleri H\"{u}gel, F. Marks, J. Vit\'oria; \emph{Silting modules}, Int. Math. Res. Not. IMRN
2016 (4), 1251–1284.
\bibitem[AMV2]{tau-tilt} L. Angeleri H\"{u}gel, F. Marks, J. Vit\'oria; \emph{ A characterisation of $\tau-$tilting finite algebras},
Model Theory of Modules, Algebras and Categories, Contemporary Mathematics 730,
Amer. Math. Soc., Providence, RI, (2019), pp. 75-89
\bibitem[AP]{wideTorsion} S. Asai, C. Pfeifer; \emph{Wide subcategories and lattices of torsion classes}, preprint,  arXiv:1905.01148v2
\bibitem[A]{AusThm} M. Auslander; \emph{ Large modules over artin algebras}, in Algebra, Topology and Category Theory: a collection of papersin honor of Samuel Eilenberg; Academic Press (1976), 1–17.
\bibitem[BCZ]{minimIncl} E. Barnard, A. Carroll, S. Zhu; \emph{Minimal inclusions of torsion classes},  Algebraic Combinatorics, Volume 2 (2019) no. 5, p. 879-901 
\bibitem[BP]{cosiltMod} S. Breaz, F. Pop; \emph{Cosilting modules}, Algebras and Representation Theory, 20, pp. 1305-1321 (2017).
\bibitem[CB1]{locallyFp} W. Crawley-Boevey; \emph{Locally finitely presented additive categories}; Communications in Algebra, 22(5), 1641-1674 (1994).
\bibitem[CB2]{infDimInRT} W. Crawley-Boevey; \emph{Infinite-dimensional modules in the representation theory of finite-dimensional algebras}, Canadian Math. Soc. Conf. Proc., 23 (1998), 29-54. 
\bibitem[D]{dualities} S. Dean; \emph{Duality and contravariant functors in the representation theory of Artin algebras}, J. Algebra Appl. 18(2019), no. 6, 27 pp.
\bibitem[DIJ]{g-vectors} L. Demonet, O. Iyama, G. Jasso; \emph{ $\tau$-tilting finite algebras, bricks, and g-vectors }, Int. Math. Res. Not. IMRN (2019) no. 3, 852–892
\bibitem[DIRRT]{lattices} L. Demonet, O. Iyama, N. Reading, I. Reiten, H. Thomas; \emph{Lattice theory of torsion classes},  arXiv:1711.01785v2.
\bibitem[E]{monobrick} S. Enomoto; \emph{Monobrick, a uniform approach to torsion-free classes and wide subcategories}, preprint, arXiv:2005.01626v1.
\bibitem[HRS]{HRS-tilt} D. Happel, I. Reiten, S.O. Smal\o ; \emph{Tilting in abelian categories and quasitilted algebras}, Mem. Amer. Math. Soc.120 (1996), no. 575.
\bibitem[He]{slides} I. Herzog; \emph{The Maximal Ziegler Spectrum of a Cotilting
Class}, slides of a talk given at the conference ``Some trends in Algebra'', Prague, (2009).
\bibitem[MS]{wideLoc} F. Marks, J. \v{S}tov\'{i}\v{c}ek; \emph{Torsion classes, wide subcategories and localisation}, Bulletin of the London Mathematical Society, Volume 49 (2017) Issue 3, p. 405-416
\bibitem[PS]{GrothendieckHearts} C. Parra, M. Saor\'{i}n; \emph{ Direct limits in the heart of a t-structure: the case of a torsion pair}, Journal of Pure and Applied Algebra, 219(9):4117–4143, (2015).
\bibitem[PSV]{finiteHearts} C. Parra, M. Saor\'{i}n, S. Virili; \emph{Locally finitely presented and coherent hearts}, preprint, arXiv:1908.00649v1.
\bibitem[Pr]{purity} M. Prest; \emph{Purity, spectra and localisation}, Encyclopedia of Mathematics and its Applications 121, Cambridge University Press, (2009).
\bibitem[Ra]{RapaThesis} A. Rapa; \emph{ Simple objects in the heart of a t-structure }, Ph. D. thesis, Universiy of Trento, (2019).
\bibitem[Ri]{Ringel}  C. M. Ringel;  \emph{Infinite dimensional representations of finite dimensional hereditary algebras}, Symposia Mathematica 23, Academic Press, London, pp. 321–412, (1979).
\bibitem[ST]{BTconj} S. Schroll, H. Treffinger; \emph{ A $\tau-$tilting approach to the first Brauer-Thrall conjecture}, preprint,  arXiv:2004.14221v3.
\bibitem[S]{tPairTilt} S. O. Smal\o; \emph{Torsion theories and tilting modules}, Bull. London Math. Soc., 16, pp. 518 - 522,(1984).  
\bibitem[V]{QvS} J. Vit\'{o}ria; \emph{Quantity vs size in representation theory}, preprint, arXiv:2001.04730.
\bibitem[WZ]{cosilt} J. Wei, P. Zhang; \emph{Cosilting complexes and AIR-cotilting modules}, Journal of Algebra 491, pp. 1-31, (2017).
\end{thebibliography}
\end{document}